\documentclass{amsart}
\usepackage{epsfig}
\usepackage[utf8]{inputenc} 
\usepackage[all]{xy}
\usepackage{amssymb}
\usepackage{amsmath}
\newtheorem{theorem}{Theorem}[section]
\newtheorem{lemma}[theorem]{Lemma}
\newtheorem{definition}[theorem]{Definition}
\newtheorem*{remark}{Remark}{\it}{\rm}

\newtheorem{corollary}[theorem]{Corollary}

\newtheorem*{bem}{Remark}{\it}{}

\numberwithin{equation}{section}

\newcommand{\Z}{{\mathbb Z}}

\newcommand{\rk}{{\rm rk}}
\newcommand{\hyp}{{\rm hyp}}
\newcommand{\an}{{\rm an}}

\newcommand{\N}{{\mathbb N}}

\newcommand{\Lmin}[1]{\Lambda_{#1}^{{\min}}}

\newcommand{\Lmax}[1]{\Lambda_{#1}^{{\max}}}

\newcommand{\Li}[1]{\Lambda_{#1}^{(i)}}
\newcommand{\Lj}[1]{\Lambda_{#1}^{(j)}}
\newcommand{\Lk}[1]{\Lambda_{#1}^{(k)}}
\newcommand{\Ljp}[1]{\Lambda_{#1}^{(j+1)}}
\newcommand{\Lip}[1]{\Lambda_{#1}^{(i+1)}}

\newcommand{\Lkp}[1]{\Lambda_{#1}^{(k+1)}}

\begin{document}

\title[Hyperbolic bases]
{Common hyperbolic bases for chains of alternating or quadratic lattices}  
\author[R. Schulze-Pillot]{Rainer Schulze-Pillot}
\thanks{To appear in Mathematische Annalen,  DOI:
  10.1007/s00208-019-01824-5. The final publication is available at
  Springer via https://link.springer.com/journal/208}%/f/a/XLjVWTzznQYtTX0npZp8Hg~~/AABE5gA~/RgRegiSDP0QZaHR0cDovL2R4LmRvaS5vcmcvW2luc2VydFcDc3BjQgoAAAPxoFxl6YvgUhdzY2h1bHplcEBtYXRoLnVuaS1zYi5kZVgEAAAG5w~~
%  DOI: 10.1007/s00208-019-01824-5}

%\thanks{MSC 2000: Primary 11E12, Secondary 11F27 11F30 11F46 11E45}
  \begin{abstract} We give a short and purely bilinear proof of the fact that
    two chains of $p$-elementary lattices with quadratic form or
    alternating bilinear form over $\Z_p$ or more generally over a
    complete discrete valuation ring 
    have common hyperbolic bases. This fact, 
    which is useful for the study of Bruhat-Tits buildings, 
    has been proven before with different methods by Abramenko and Nebe
    and by Frisch. 
 \end{abstract}

\maketitle

\section{Introduction.} 
Let $R$ be a complete discrete valuation ring with field of fractions
$F$, $p$ a prime element in $R$.  
Let $V$ be a finite dimensional vector space over $F$ with a non
degenerate alternating or symmetric bilinear form $b$. In the
symmetric case we assume $b$ to be associated to a  quadratic form
$Q$ on $V$ satisfying $b(x,y)=Q(x+y)-Q(x)-Q(y)$, we have then
$b(x,x)=2Q(x)$ for all $x \in V$ .
In the alternating case we set
$Q\equiv 0$.
A vector $x \in V$ is called isotropic if $x\ne 0, Q(x)=0$, a nonzero
subspace or $R$-submodule $X$ is called totally isotropic if $Q(x)=b(x,y)=0$ holds
for all $x,y \in X$, it is called anisotropic if $Q(x)\ne 0$ for all
$x\in X\setminus \{0\}$.
An $R$-lattice $\Lambda$ of maximal rank on $V$
%is called integral if $Q(\Lambda)\subseteq R$ and 
%$b(\Lambda,\Lambda)\subseteq R$ holds, 
is called $p^r$-maximal if it is maximal among the lattices satisfying
$Q(\Lambda)\subseteq p^rR$ and $b(\Lambda,\Lambda)\subseteq p^rR$, for
$r=0$ we say that the lattice is maximal. If $(V,Q)$ is anisotropic
there is a unique
$p^r$-maximal lattice on $V$, namely $\{x\in V\mid Q(x) \in p^rR\}$, see
\cite[16.1]{kneserbuch}.

It is well known (see e.g. \cite{satake}, \cite {eichler_qfog} for the
symmetric case) that any two  $p^{r_1}$ resp. $p^{r_2}$ maximal
lattices $\Lambda_1, \Lambda_2$ have a common hyperbolic basis, i.e.,
there are vectors $\{e_i,f_i\}$ of $V$ with $b(e_i,e_j)=b(f_i,f_j)=0,
b(e_i,f_j)=\delta_{ij}$ such that suitable multiples of the $e_i,f_i$
together with a basis of a maximal lattice on an anisotropic kernel of
$V$ form bases of $\Lambda_1$ and $\Lambda_2$.  As explained in
\cite{garrettbook}, in the theory of
Bruhat-Tits buildings one needs the even sharper statement that such
common bases exist for pairs of certain chains of lattices. Proofs of
such statements have been given in \cite{abramenko_nebe} using the
theory of hereditary orders and in \cite{frisch} using the concept of
$p$-adic norms from \cite{bruhat-tits}. The purpose of this note is to
give a short elementary proof of this fact using only the quadratic  and bilinear
forms. It is easy to generalize our argument to the case of hermitian
forms, we leave the details of this to the reader.  

\section{Lattice chains}
\begin{definition}
Let $\Lambda$ be a lattice on $V$ (or a subset of a lattice spanning
$V$). The dual lattice of $\Lambda$ is 
$\Lambda^\#:=\{x \in V\mid b(x,\Lambda)\subseteq R\}$. The lattice is
called $p$-elementary if $\Lambda^\#\supseteq \Lambda\supseteq
p\Lambda^\#$ holds. It is called $p^r$-modular if
$\Lambda=p^r\Lambda^\#$ holds. 
\end{definition}
\begin{remark}
It is well known that $\Lambda^\#$ is a lattice on $V$ and that
$(\Lambda^\#)^\#=\Lambda$ holds. The lattice is $p$-elementary if and
only if $pR\subseteq b(x,\Lambda)\subseteq R$ holds for all primitive vectors
$x \in \Lambda$.   
\end{remark}

We need a few preparations for the case of a symmetric bilinear
form. % We recall that an $R$- lattice $\Lambda$ on the non degenerate
% quadratic space $(V,Q)$ over $F$ is called totally even if
% $Q(x)\in b(x,\Lambda)$ holds for all $x \in
% \Lambda$. Equivalently, the $P^j$-modular Jordan component $L_j$ of
% $\Lambda$ satisfies $Q(L_j)\subseteq P^j$ for all $j$. 
\begin{definition}
 Let $(V,Q)$ be a regular quadratic space over $F$ and $\Lambda$ a
 lattice on $V$.
 \begin{enumerate}
 \item $\Lambda$ is called integral if $b(\Lambda,\Lambda)\subseteq
   R$, even if $Q(\Lambda)\subseteq R$, totally even if it is integral
   and $Q(x)\in b(x,
   \Lambda)$ holds for all $x \in \Lambda$.
   \item We say that the lattice $\Lambda$ is almost $p$-elementary
     totally even (or of $p$-elementary totally even type) if $V$ has
     a Witt decomposition $V=V^\hyp\perp V^\an$, where $V^\hyp$ is a sum of
     hyperbolic planes and $V^\an$ is anisotropic, such that  
     $\Lambda=\Lambda\cap V^\hyp\perp \Lambda\cap V^\an$ with
     $\Lambda^\hyp:=\Lambda\cap V^\hyp$  totally
     even and $p$-elementary and $\Lambda^\an:=\Lambda \cap V^\an=\{x\in
     V^\an\mid Q(x)\in R\}$  the unique $R$-maximal
     lattice on $V^\an$.
     \item If $\Lambda$ is almost $p$-elementary totally even, the
       modified dual $\Lambda^\ast$ is $\{x\in \Lambda^\# \mid
       pQ(x)\in R\}$ with the quadratic form $Q^\ast:=pQ$ and the
       bilinear form $b^\ast:=pb$.
 \end{enumerate}
\end{definition}
\begin{remark}
  \begin{enumerate}
    \item Part c) of this definition is  a slight modification of the
      one given by Frisch \cite{frisch},
      who also proved a version of the next two lemmata.
      \item In particular, a maximal lattice on $V$ is almost $p$-elementary
totally even.  
  \item If we set $Q=0$ for an alternating bilinear form $b$, the
    definition of almost $p$-elementary totally even above coincides
    with $p$-elementary. We will use ``almost $p$-elementary totally
    even'' for both types of $b$ in what follows.

    We also use some more terminology from the symmetric case in the
alternating case as well.  In particular we generalize the notion of dual
lattice and $p^r$-modular lattice to the alternating case in the
obvious way and call a lattice $Rx+Ry$ with
$Q(x)=Q(y)=0, b(x,y)=p^r$ a $p^r$-modular hyperbolic plane.

  \end{enumerate}
\end{remark}

\begin{lemma}\label{modified_doubledual}
  The modified dual $(\Lambda^\ast, Q^\ast)$ of an almost $p$-elementary totally even lattice
  is  almost
  $p$-elementary totally even and one has
  \begin{equation*}
    ((\Lambda,Q)^\ast)^\ast=(p^{-1}\Lambda, p^2Q)\cong (\Lambda, Q).
  \end{equation*}
\end{lemma}

\begin{proof}
We write $\Lambda=\Lambda^\hyp\perp\Lambda^\an$ and have
$\Lambda^\#=(\Lambda^\hyp)^\#\perp (\Lambda^\an)^\#$ with $pQ(x) \in
R$ for all $x \in (\Lambda^\hyp)^\#$ and therefore
$\Lambda^*=(\Lambda^\hyp)^\#\perp \{ x \in (\Lambda^\an)^\#\mid
pQ(x)\in R\}$. Obviously,
$(\Lambda^\hyp)^\#$, equipped with $pQ, pb$, is $p$-elementary totally
even. For the anisotropic part, the set  $\{x\in (\Lambda^\an)^\#\mid
pQ(x)\in R\}$ is a lattice which clearly is contained in the unique maximal lattice on
$(V^\an, pQ)$. For the reverse inclusion, let $x \in V^\an$ with $Q(x)
\in p^{-1}R$. If one had $x\not\in (\Lambda^\an)^\#$ 
there would exist $y \in \Lambda^\an$ with $b(x,y)=p^{-1}$, hence
$Q(x)=-b(x,ay)$ with $a \in R$. This gives $Q(x+ay)\in R$, hence $x+ay
\in \Lambda^\an\subseteq (\Lambda^\an)^\#$, which contradicts $x\not\in
(\Lambda^\an)^\#$, and we see that the maximal lattice on
$(V^\an, pQ)$ is indeed contained in  $\{x\in (\Lambda^\an)^\#\mid
pQ(x)\in R\}$.

Finally, the same argument shows that
$(\Lambda^*,pb)^*=p^{-1}\Lambda^\hyp\perp \{x\in V^\an\mid p^2Q(x)\in
R\}=p^{-1}\Lambda$ with $p^{-1}\Lambda$ equipped with $p^2Q$.
\end{proof}
In what follows we will identify $(p^{-1}\Lambda,p^2Q)$ with
$(\Lambda,Q)$ and will therefore write
$(\Lambda^\ast)^\ast=\Lambda$. We will also treat both cases of $b$
(symmetric or alternating) at the same time in what follows.
\begin{lemma}
Let $\Lambda$ be an almost $p$-elementary totally even lattice on $V$
and let $J\subseteq\Lambda$ be a unimodular or $p$-modular hyperbolic
plane which splits $\Lambda$, i.e., $\Lambda=J\perp \Lambda_1$. Then
$\Lambda_1$ is almost $p$-elementary totally even.  
\end{lemma}
\begin{proof}
By definition we have $\Lambda=\Lambda'\perp\Lambda''$, where
$V^\hyp=V'=F\Lambda'$ is a sum of hyperbolic planes,
$V^\an=V''=F\Lambda''$ is 
anisotropic, and $\Lambda'$ is totally
even $p$-elementary on $V'$.
$\Lambda'$ has a splitting $\Lambda'=K_1\perp \dots\perp K_s$  into
binary lattices $K_i$ which are unimodular or $p$-modular hyperbolic
planes. At least one of the $K_i$ is unimodular if and only if there
exists an isotropic vector $x\in \Lambda$ with $b(x, \Lambda)=R$, in
particular, if $J$ is unimodular, at least one of the $K_i$, say
$K_1$, is unimodular. By \cite[Folgerung 4.4]{kneserbuch} there exists $\sigma
\in O(\Lambda)$ with $\sigma(J)=K_1$ and hence
$\sigma(\Lambda_1)=K_1^\perp$, the latter one being almost
$p$-elementary totally even. Hence $\Lambda_1$ is almost $p$-elementary
totally even too.

If $J$ is $p$-modular, $J^\ast$ is unimodular and splits
$\Lambda^\ast$. By the first case, the orthogonal complement
$\Lambda_2$ of
$J^\ast$ in $\Lambda^\ast$ is almost $p$-elementary totally even, which implies that
$\Lambda_1=(\Lambda_2)^\ast$ is almost $p$-elementary totally even.
\end{proof}
\begin{lemma}\label{properties_p-elementary}
Let $\Lambda=\Lambda^{\hyp}\perp\Lambda^\an$ be an almost $p$-elementary totally even lattice
containing isotropic vectors.
\begin{enumerate}
\item If the hyperbolic part $\Lambda^\hyp$ of $\Lambda$ is not a sum
  of 
  $p$-modular hyperbolic planes or $\Lambda=\Lambda^\hyp$ holds, $\Lambda$ is generated
  by its isotropic vectors.
\item If $\Lambda^\hyp$ is a sum of $p$-modular hyperbolic planes, the
  isotropic vectors of $\Lambda$ generate the sublattice $\{x\in
  \Lambda\mid Q(x) \in pR\}$.
  \item One has %$\{x \in p\Lambda^\#\mid Q(x)\in R\} \subseteq
    % \Lambda$ and
    $\{z\in p\{x\in \Lambda \mid Q(x)\in pR\}^\#\mid Q(z)\in R\}\subseteq \Lambda$.
\end{enumerate}
\end{lemma}
\begin{proof}
The first two assertions follow from the fact that $\Lambda^\hyp$ is
obviously generated by its isotropic vectors and that for any vector $x
\in \Lambda$ satisfying $Q(x)\in Q(\Lambda^\hyp)$ one can find $y\in
\Lambda^\hyp$ with $Q(x+y)=0$.
For the last part of the lemma the set $ \{x\in \Lambda \mid Q(x)\in
pR\}$ generates $\Lambda$ if $\Lambda^\hyp$ is not a sum of $p$-modular
hyperbolic planes and is equal to $ \{x\in \Lambda \mid Q(x)\in
pR\}=\Lambda^\hyp\perp \{x\in V^\an \mid Q(x)\in pR\}$
otherwise. Taking duals and multiplying by $p$ we obtain the assertion.
\end{proof}

\begin{definition}\label{lattice_chain} Let $n$ denote the Witt index
  of $V$, i.e., the dimension of a maximal totally isotropic subspace.
  
 A maximal admissible lattice chain ${\mathcal L}$ in $(V,Q)$ is a chain of lattices
 $\Lambda^{\max}=\Lambda^{(0)}\supsetneq \Lambda^{(1)}\dots \supsetneq
 \Lambda^{(n)}=\Lambda^{(\min)}$, where $\Lambda^{(0)}$ is a maximal lattice on $V$,
 $n$ is the Witt index of $V$, and each $\Lambda^{(j)}$ is almost 
   $p$-elementary  totally even.
\end{definition}

\begin{remark}
 The last lattice $\Lambda^{\min}$ of a maximal admissible lattice chain is the
 orthogonal sum of $p$-modular hyperbolic planes  and a maximal
 lattice on an anisotropic space, whereas in $\Lambda^{\max}$ the hyperbolic
 planes occurring are all unimodular. Moreover, we have
 $p\Lambda^{\max}\subseteq \Lambda^{\min}$.
\end{remark}

\section{Hyperbolic bases}
\begin{theorem}\label{isotropic_bases_theorem}
  Let $\Lambda\subseteq V$ be an almost $p$-elementary totally even
  $R$-lattice on $V$.

  Let $X$ be a maximal
  totally isotropic submodule of $\Lambda$. 
Then there are a basis
  $(e_1,\ldots, e_n)$ of $X$ and  vectors $f_1, \ldots,f_n
  \in \Lambda$  generating a totally isotropic submodule of $\Lambda$
  and satisfying $b(e_i,f_j)\in \{\delta_{ij},p\delta_{ij}\}$ such
  that 
\begin{equation}
\Lambda=\bigoplus_{i=1}^nRe_i\oplus
  \bigoplus_{i=1}^n Rf_i\perp K,\end{equation}
 where $K$ is the unique maximal lattice on an
  anisotropic subspace $FK$ of $V$.
\end{theorem}
\begin{proof}
We use induction on $n=\rk(X)$. For $n=0$ the space $V$ is $\{0\}$ in
the alternating case, $\{0\}$ or anisotropic in the symmetric case,
and the assertion is trivial.

Let $n\ge 1$ and assume the assertion to be true for $\rk(X)<n$.

If $b(X,\Lambda)=R$ we choose $x\in X, y\in \Lambda$ with
$b(x,y)=1$. Replacing $y$ by $y-Q(y)x$ if necessary we may assume $y$
to be isotropic so that $Rx+Ry$ is a unimodular hyperbolic plane. We
can then split $\Lambda$ as $(Rx+Ry)\perp \Lambda'$, with
$\Lambda'\cap X=\{z \in X \mid b(z,y)=0\}$ maximal totally isotropic
in $\Lambda'$ of rank $n-1$, and the
induction hypothesis implies the assertion in this case.

If $b(X,\Lambda)=pR$ we have $p^{-1}X\subseteq \Lambda^\ast$ with
$b^\ast(p^{-1}x, \Lambda^\ast)=R$ for all primitive vectors $x\in X$,
in particular we see that $p^{-1}X$ is primitive and hence maximal
totally isotropic in $\Lambda^\ast$.

Using the first case we find
a basis
  $(p^{-1}e_1,\ldots, p^{-1}e_n)$ of $p^{-1}X$ and  vectors $f_1^\ast, \ldots,f_n^\ast
  \in \Lambda^\ast$  generating a totally isotropic submodule of $\Lambda^\ast$
  and satisfying $b(e_i,f_j^\ast)\in \{\delta_{ij},p\delta_{ij}\}$ such
  that 
\begin{equation}
\Lambda^\ast=\bigoplus_{i=1}^nRp^{-1}e_i\oplus
  \bigoplus_{i=1}^n Rf_i^\ast\perp K,\end{equation}
 where $K$ is the unique maximal lattice on an
 anisotropic subspace $FK$ of $V$.
Since we have $b(e_i, \Lambda^\ast)=b^\ast(p^{-1}e_i,\Lambda^\ast)=R$
for all $i$ we see that the case $b(e_i,f_i^\ast)=p$ can not occur.
 We set $f_i=pf_i^\ast$ for all $i$ and 
  obtain the assertion by taking the modified dual of both sides of
  the last equation.
\end{proof}
An obvious consequence is:
\begin{corollary}\label{orbits}
  With $\Lambda, b, Q$ as in the proposition all maximal totally
  isotropic submodules of $\Lambda$ are in the same orbit under the
  action of the isometry group of $(\Lambda, b, Q)$. 

In particular, in the alternating case the symplectic group of
$(\Lambda,b)$ (also called a local paramodular group of level $p$ if
$\Lambda$ has both unimodular and $p$-modular components )
acts transitively on the set of maximal totally 
isotropic submodules, a fact which is well known for the integral
symplectic group.
\end{corollary} 

\begin{theorem}
  Let $\Lmax\nu=\Lambda_\nu^{(0)}\supseteq \dots \supseteq
  \Lambda_\nu^{(n)}=\Lmin\nu$ for $\nu=1,2$ be two maximal admissible lattice chains on $V$.

  Then there exist isotropic vectors
 $e_1,\ldots,e_n,f_1,\ldots,f_n \in \Lambda^{(0)}_1$ with
 $b(e_i,e_j)=0=b(f_i,f_j), b(e_i,f_j)=\delta_{ij}$
 such that 
 \begin{equation}
   \label{eq:2}
  \Lambda^{(0)}_1=\bigoplus_{i=1}^nRe_i \oplus \bigoplus_{i=1}^nRf_i\perp
   K,
 \end{equation}
where $K$ is the unique maximal  lattice on the anisotropic orthogonal
complement of the space generated by the $e_i,f_i$, and such that
\begin{equation}
  \label{eq:3}
 \Lambda_\nu^{(j)}=\bigoplus_{i=1}^nRp^{r_i^{(\nu,j)}}e_i \oplus \bigoplus_{i=1}^nRp^{s_i^{(\nu,j)}}f_i\perp
   K
\end{equation}
holds with certain integers $r_i^{(\nu,j)},s_i^{(\nu,j)}$ for
$\nu=1,2$ and $0\le j \le n$.
\end{theorem}
\begin{proof} We prove the assertion by induction on the Witt index
  $n$ of $V$. The case $n=0$ is trivial, so we assume for the rest of
  the argument $n>0$. In that case 
  we notice first that $\Lmax\nu$ is generated by its isotropic vectors
  by Lemma \ref{properties_p-elementary}
  and that $b(x,\Lmax\nu)=R$ holds for each primitive isotropic vector
  $x\in \Lmax\nu$. On the other hand, $b(x,\Lmin\nu)= pR$ holds for all
  primitive isotropic $x\in \Lmin\nu$.
  
There exists $r\in \N$ with $p^r\Lmax1\subseteq \Lmin2,
p^r\Lmax2\subseteq \Lmin1$, without loss of generality we may assume
$p^{r-1}\Lmax1\not \subseteq \Lmin2$. Since $\Lmax1$ is generated by
isotropic vectors we find $x \in \Lmax1$ isotropic such that $p^rx$ is primitive
in $\Lmin2$. As noticed above we must have $b(p^rx, \Lmin2)=pR$.

Assume that $x$ can be chosen such that $p^rx$ is primitive in
$\Lmax2$.
One has  then $b(p^rx, \Lmax2)=R$, and there is $0\le j<n$
with $b(p^rx, \Lj2)=R, b(p^rx,\Ljp2)=pR$.
We choose then $x$ with $p^rx$ primitive in $\Lmax2$ such that the largest integer $k$ with $x\in \Lk1$ is as
large as possible.

%We choose the isotropic vector
%$x\in \Lmax1$ so that this $j$ is maximal

Choose $y\in \Lj2$ with $b(p^rx,y)=1$, we have then $py$ primitive in
$\Ljp2$ and $py \in \Lmin2$, moreover $b(py, \Lmax2)=pR$ and hence
$b(Rp^rx+Rpy, \Li2)=pR$ for all $i>j$. This implies that $Rp^rx+
Rpy$ splits off orthogonally in $\Ljp2,\ldots, \Lmin2$, whereas
$Rp^rx+Ry$ splits off orthogonally in $\Lmax2,\ldots,\Lj2$.

If $x$
could be chosen to be in $\Lmin1$ we can split off $Rx+Rp^ry$ orthogonally in
all $\Li1$ and proceed by induction on $n$.
Otherwise there is $0\le k<n$ with $x\in \Lk1$ and  $px$ primitive in $\Lkp1$. With $y$ as
above the unimodular hyperbolic plane $Rx+Rp^ry$ splits off
orthogonally in $\Lmax1,\ldots,\Lk1$. If $b(p^ry, \Lkp1)=pR$ holds,
$Rpx+Rp^ry$ splits off orthogonally in $\Lkp1,\ldots,
\Lmin1$, and we can proceed by induction. Otherwise there exists $x'\in \Lkp1$ with $b(p^rx',
\Lmax2)=R$, hence $p^rx'$  primitive in $\Lmax2$, which contradicts
our choice of $x$, so this situation can not occur and we are done with the
case that $p^rx$ can be chosen to be primitive in $\Lmax2$.

\medskip

We are left with the case that $p^rx$ can not be chosen to be primitive in $\Lmax2$, hence
$p^{r-1}x \in \Lmax2$ for all isotropic $x \in \Lmax1$, which implies
$p^{r-1}\Lmax1\subseteq \Lmax2$. If one has
$p^{r-1}\Lmax2\not\subseteq \Lmax1$ we may interchange the two chains
and reduce to the previous case, so we may assume that
$p^{r-1}\Lmax2\subseteq \Lmax1$ holds as well.

Since $p^{r-1}\Lmax1 \not\subseteq \Lmin2$ holds by assumption and
$p^{r-2}x \in \Lmax2$ implies $p^{r-1}x \in \Lmin2$  we
 can choose $x \in\Lmax1$ isotropic with
$p^{r-1}x$ primitive in $\Lmax2$. Among such $x$ we choose one for which 
 the largest integer $k$ with $x\in \Lk1$ is maximal.  If we have
 $k=n$ there exists $0\le i<n$ such that we can choose $x\in p
 (\Lip1)^\ast \subseteq \Lmin1$
 and we impose the additional condition that the least integer $i$
  with $x\in p(\Lip1)^\ast$ is minimal. We fix these integers $i,k$
  for the rest of the proof.
 
We notice that if $k<n$
holds for the maximal $k$ above, $p^{r-2}x \in \Lmax2$ holds for all
isotropic $x\in \Lmin1$, and we obtain
\begin{equation*}p^{r-1}\Lmax2= p^{r-1}\{z\in (\Lmax2)^\#\mid Q(z)\in
  R\}\subseteq  p\{x\in \Lmin1\mid Q(x)\in
  pR\}^\# \subseteq \Lmin1 
\end{equation*} by dualizing and applying Lemma \ref{properties_p-elementary}.
Moreover, if $k=n$ holds we have similarly
\begin{equation*}
  p^{r-2}p (\Li1)^\ast\subseteq \Lmax2, \quad p^{r-1}\Lmax2\subseteq \Li1. 
\end{equation*}

If we have $p^{r-1}x \in \Lmin2$  the primitivity of $p^{r-1}x$ in
$\Lmax2$ implies $b(p^{r-1}x, \Lj2)=R,
b(p^{r-1}x,\Ljp2)=pR$ for some $0\le j <n$. 
We choose  $y\in \Lj2$ isotropic with $b(p^{r-1}x,y)=1$ and have
$py$ primitive in $\Ljp2,\ldots, \Lmin2$. The unimodular hyperbolic
plane $Rp^{r-1}x+Ry$ then splits off orthogonally in
$\Lmax2,\ldots,\Lj2$ and the $p$-modular hyperbolic plane $Rp^{r-1}x+Rpy$
splits off orthogonally in $\Ljp2,\ldots,\Lmin2$.

\medskip
If on the other hand $p^{r-1}x \not\in \Lmin2$ holds there
exists $0\le j<n$ with $p^{r-1}x \in \Lj2, p^{r-1}x\not\in \Ljp2$.
Since $p^rx$ is primitive in $\Ljp2$ with $b(p^rx,\Ljp2)=pR$, we have
$p^rx\in p(\Ljp2)^\ast$ primitive and find isotropic $y'\in (\Ljp2)^\ast$ with
$b^\ast(p^{r-1}x, y')=1$. 
With $y=py'\in \Lmin2$ we have $b(p^rx,y)=p$ and $b(y, \Ljp2)=pR$, and
the $p$-modular hyperbolic plane $Rp^rx+Ry$ splits off orthogonally in
$\Ljp2,\ldots, \Lmin2$, whereas the unimodular hyperbolic plane
$Rp^{r-1}x+Ry$ splits off orthogonally in $\Lmax2,\ldots,\Lj2$.

In both cases, by our assumptions we have $p^{r-1}y \in \Lmax1$. If in addition
$p^{r-1}y\in \Lmin1$ 
holds, the unimodular hyperbolic plane $Rx+Rp^{r-1}y$ splits off
orthogonally in $\Lmax1,\ldots,\Lk1$. If we had $b(p^{r-1}y, \Lkp1)=R$
this would imply $k+1<n$ and the existence of a vector $x'\in \Lkp1$
with $b(p^{r-1}x',\Lmax2)=R$, which contradicts the maximality
property of $k$. The $p$-modular hyperbolic plane $Rpx+Rp^{r-1}y$
therefore splits off orthogonally in $\Lkp1,\ldots, \Lmin1$.

If $p^{r-1}y \not \in \Lmin1$ holds we have seen that we must have
$k=n$ and hence $x\in p(\Lip1)^\ast\subseteq \Lmin1$ and $b(x,
\Li1)=R$, $p^{r-1}\Lmax2 \subseteq (\Li1)$, hence
$p^{r-1}y\in \Li1$. From $b(x,p^{r-1}y)=1$ we see that $y\not\in
\Lip1$ so that $p^ry$ is primitive in $\Lip1,\ldots, \Lmin1$.
 
The unimodular hyperbolic plane $Rx+Rp^{r-1}y$ then splits  off orthogonally in
$\Lmax1,\ldots, \Li1$, whereas the $p$-modular hyperbolic plane
$Rx+Rp^ry$ splits off orthogonally in $\Lip1,\ldots,\Lmin1$, 
and  we can again reduce the Witt
index $n$ by $1$ and obtain the assertion from the induction hypothesis.
\end{proof}

%\enlargethispage{2cm}

\medskip
Rainer Schulze-Pillot\\
Fachrichtung Mathematik\\
Universit\"at des Saarlandes\\
Postfach 151150, 66041 Saarbr\"ucken, Germany\\
email: schulzep@math.uni-sb.de


\begin{thebibliography}{MVW}
\bibitem{abramenko_nebe} P. Abramenko, G. Nebe:
  Lattice chain models for affine buildings of classical type,
Math. Ann. {\bf 322} (2002), no. 3, 537–562. 
\bibitem{bruhat-tits}  F. Bruhat, F, J. Tits:  Schémas en groupes et
  immeubles des groupes classiques sur un corps local. II. Groupes
  unitaires, Bull. Soc. Math. France {\bf 115} (1987), no. 2, 141–195.
\bibitem{eichler_qfog} M. Eichler: Quadratische Formen und Orthogonale
  Gruppen, Grundlehren d. math. Wiss. 63, Springer Verlag 1952.
\bibitem{frisch} W. Frisch: The cohomology of $S$-arithmetic spin
  groups and related Bruhat-Tits buildings, doctoral dissertation Göttingen 2002.
\bibitem{garrettbook} P. Garrett: Buildings and Classical Groups,
  Chapman and Hall 1997.
 \bibitem{kneserbuch} M. Kneser: Quadratische Formen, Springer Verlag 2002.
\bibitem{satake} I. Satake: Theory of spherical functions on reductive
  algebraic groups over ${\mathfrak p}$-adic fields. Inst. Hautes
  Études Sci. Publ. Math. {\bf 18} 1963, 5–69. 
  \end{thebibliography}
\end{document}